\documentclass[12pt]{article}

\usepackage{amsmath,amssymb,amsfonts}
\usepackage{amsthm}%
\usepackage{enumitem}
\usepackage{graphicx}
\usepackage{mathrsfs}
\usepackage{enumerate}

\allowdisplaybreaks%

{\theoremstyle{plain}%
 \newtheorem{theorem}{Theorem}
 
 \newtheorem{lemma}{Lemma}%
}
{\theoremstyle{remark}
\newtheorem{remark}{Remark}
}
{\theoremstyle{definition}

}

\begin{document}

\begin{center}
 {\large  On a conjecture concerning antipodal labelings for cycles}

 \ 

{\sc John M. Campbell}

\vspace{0.1in}

{\footnotesize Department of Mathematics and Statistics}

{\footnotesize Dalhousie University}

{\footnotesize Halifax, NS B3H 4R2}

{\footnotesize Canada}

\vspace{0.1in}

{\footnotesize {\tt jh241966@dal.ca}}

\vspace{0.1in}

\end{center}

\begin{abstract}
 Let $G$ denote a finite, connected, simple graph, and let $D$ denote its diameter, and let $d_{G}(u, v)$ denote the distance between 
 vertices $u$ and $v$ in $V(G)$. An \emph{antipodal labeling} of $G$ is a mapping $f\colon V(G) \to \mathbb{N}_{0}$ such that, for each 
 pair $(u, v)$ consisting of distinct vertices in $V(G)$, the relation $ |f(u) - f(v)| \geq D - d_{G}(u, v) $ holds. The \emph{span} of $f$ is then 
 defined so that $\operatorname{sp}(f) = \max\{ f(u) - f(v) : u, v \in V(G) \}$. The \emph{antipodal number} of $G$, denoted with 
 $\operatorname{an}(G)$, may then be defined as the minimum possible span among all antipodal labelings of $G$. Juan and Liu 
 [\emph{Ars Combin.}, 2012] proved the upper bound $\operatorname{an}(C_{4k}) \leq 2k^2 - 1$, and conjectured that 
 $\operatorname{an}(C_{4k}) = 2k^2 - 1$ for each positive integer $k$, and Juan and Liu verified this conjecture for $k \leq 5$. We 
 succeed in proving Juan and Liu's conjecture, which seems to have remained open. The openness of Juan and Liu's conjecture has been 
 noted, over the years, by many authors, including Rao et al.\ [\emph{Contrib. Discrete Math.}, 2015] and Saha et al.\ [\emph{Theory 
 Comput. Syst.}, 2022]. 
\end{abstract}

\vspace{0.1in}


\vspace{0.1in}


 \section{Introduction}
 The purpose of this paper is to prove a conjecture due to Juan and Liu \cite{JuanLiu2012} concerning antipodal labelings, as defined 
 below, for cycles. It appears that Juan and Liu's conjecture has remained open, in view of the many previous references citing their work 
 \cite{BasuniaDasSahaTiwary2021,BloomfieldLiuRamirez2022,DasSahaTiwary2020,Gallian1998,
GomathiVenugopal2022,KumarPanigrahi2026,RaoKolaPanigrahi2015,Saha2020,Saha2021,SahaBasuniaDasTiwary2022,
SahaDasDasTiwary2020,SahaPanigrahi2015}. 
 Indeed, Juan and Liu's conjecture being open has been noted by many authors over the years, including Rao et al.\ in 2015 
 \cite{RaoKolaPanigrahi2015}, Das et al.\ in 2020 \cite{DasSahaTiwary2020}, Saha et al.\ in 2020 \cite{SahaDasDasTiwary2020}, 
 Saha in 2021 \cite{Saha2021}, 
 and Saha et al.\ in 2022 \cite{SahaBasuniaDasTiwary2022}. 
 Our proof is largely based on our extensive interactions with GPT-5.6 Pro. 

 Throughout this paper, graphs are assumed to be finite, connected, and simple (unless otherwise indicated). The preliminaries covered 
 below are required for our purposes. A given graph $G$ is endowed with its usual metric $d(u, v) = d_{G}(u, v)$, so that $d(u, v)$ is the 
 number of edges in the smallest path connecting $u$ and $v$. The \emph{diameter} $\operatorname{diam}(G)$ of a graph $G$ refers 
 to the maximum distance among all pairs of vertices in $G$. 
 An \emph{antipodal labeling} of $G$ is a mapping 
 $f\colon V(G) \to \{ 0, 1, \ldots \}$
 such that 
\begin{equation}\label{antipodalcondition}
 |f(u) - f(v)| \geq \operatorname{diam}(G) - d(u, v)
\end{equation}
 holds for distinct vertices $u$ and $v$ in $G$. For such a labeling $f$, the \emph{span} of $f$ is then defined so that 
\begin{equation}\label{displayspan} 
 \operatorname{sp}(f) = \max\{ f(u) - f(v) : u, v \in V(G) \}. 
\end{equation} 
 Being consistent with the notation and terminology in the work on Juan and Liu \cite{JuanLiu2012}, 
 the \emph{antipodal number} $\operatorname{an}(G)$
 for $G$ is then defined as the minimum span for an antipodal labeling of $G$. 

 As expressed by Juan and Liu \cite{JuanLiu2012}, the concept of an antipodal labeling for a graph was introduced by Chartrand et al.\ 
 \cite{ChartrandErwinZhang2000,ChartrandErwinZhang2002}, who established general bounds for antipodal numbers. 
 Bounds for $\operatorname{an}(C_n)$ were established by Chartrand et al.\ \cite{ChartrandErwinZhang2000}, and evaluations for 
 $\operatorname{an}(C_{n})$ have been proved for each of the cases among
 $n \equiv 1 \bmod {4}$, 
 $n \equiv 2 \bmod {4}$, and 
 $n \equiv 3 \bmod {4}$ \cite{ChartrandErwinZhang2000,JuanLiu2012}.

 Juan and Liu \cite[Theorem 9]{JuanLiu2012} proved that 
\begin{equation}\label{C4kbounds}
 2k^2 - \left\lfloor \frac{k}{2} \right\rfloor \leq \operatorname{an}(C_{4k}) \leq 2k^2 - 1 
\end{equation}
 for each integer $k$ exceeding $1$, and then conjectured \cite[Conjecture 1]{JuanLiu2012} that $\operatorname{an}(C_{4k}) $ is equal to 
 the upper bound in \eqref{C4kbounds} for every positive integer $k$. Juan and Liu verified this conjecture for $k \in \{ 1, 2, \ldots, 5 \}$. 
 The conjectured evaluation for $\operatorname{an}(C_{n})$ for the $n \equiv 0 \bmod {4}$ case has remained the last unsolved case, out 
 of the congruence classes modulo $4$. 
 This gives weight to our 
 proof of Juan and Liu's conjecture in Section \ref{secmain} below. 
 
\section{Proof of Juan and Liu's conjecture}\label{secmain}
 From the upper bound in \eqref{C4kbounds} proved by Juan and Liu \cite[Theorem 9]{JuanLiu2012}, it remains to prove that the 
 reverse inequality
\begin{equation}\label{displayreverse} 
 \operatorname{an}(C_{4k}) \geq 2k^2 - 1
\end{equation}
 holds for positive integers $k$, 
 and this provides the basis for our proof of Theorem \ref{maintheorem} below. 

\begin{remark}\label{correction}
 Kumar and Panigrahi \cite{KumarPanigrahi2026} do not claim to prove Juan and Liu's conjecture, 
 but they give a related minimality criterion in their Theorem~2.4 and apply it in Theorem~2.5, 
 but it can be shown, as follows, that these results are not valid as stated. Indeed, for \(C_{8}=z_{0}z_{1}\cdots z_{7}z_{0}\), the 
 ordering \(z_{0},z_{4},z_{1},z_{5},z_{2},z_{6},z_{3},z_{7}\), with respective labels \(0,0,3,3,6,6,9,9\), satisfies the hypotheses of Theorem~2.4(a) 
 in Kumar and Panigrahi's paper, 
 as well as those of Theorem~2.5 for \(n=4\), \(d=4\), and \(m=m'=1\), although the resulting span is \(9\). By contrast, the cyclic label sequence \(0,5,2,7,0,5,2,7\) is an antipodal labeling of \(C_{8}\) with span \(7\), and Juan and Liu \cite[Theorem~9]{JuanLiu2012} give \(\operatorname{an}(C_{8})=7\).
\end{remark}

 Let $V(C_{4k}) = \{ v_{0}, v_{1}, \ldots, v_{4k-1} \}$, 
 with the understanding that subscripts are given by residue classes modulo $4k$. 
          Observe that $\operatorname{diam}(C_{4k}) = 2k$. We require 
 the following specialization of 
 a result from Juan and Liu \cite[Proposition 1]{JuanLiu2012}, 
 referring the interested reader to Juan and Liu's proof of a generalized version of the following result. 

\begin{lemma}\label{lemmadistance}
 For any three vertices $u$, $v$, and $w$ in $C_{4k}$, the relation 
 $ d(u, v) + d(v, w) + d(w, u) \leq 4 k $
 holds (cf.\ \cite[Proposition 1]{JuanLiu2012}). 
\end{lemma}

 By taking a given antipodal labeling of a graph, and by then subtracting the minimum label from each label, this does not change the 
 antipodal condition in \eqref{antipodalcondition} or the value of the right-hand side of \eqref{displayspan}. So, for an arbitrary antipodal 
 labeling $f$ of $C_{4k}$, we may assume without loss of generality that $0$ is among the labels. We proceed to fix an ordering on 
 $V(C_{4k})$, by writing 
\begin{equation}\label{VC4k} 
 V(C_{4k}) = \{ x_0, x_1, \ldots, x_{4k-1} \}, 
\end{equation}
 with 
\begin{equation}\label{orderfx} 
 0 = f(x_0) \leq f(x_1) \leq \cdots \leq f(x_{4k-1}). 
\end{equation}
 Borrowing notation from Juan and Liu \cite{JuanLiu2012}, we define $d_{i} = d(x_{i}, x_{i+1})$ and 
\begin{equation}\label{definefi} 
 f_{i} = f(x_{i+1}) - f(x_{i})
\end{equation}
 for $i \in \{ 0, 1, \ldots, 4k-2 \}$. We also require the following specialization of a lemma from Juan and Liu \cite[Lemma 2]{JuanLiu2012}. 
 Again, we refer the interested reader to Juan and Liu's paper for a derivation of a more general version of 
 the following lemma. 
 
\begin{lemma}\label{boundfsum}
 The relation $f_{i} + f_{i+1} \geq k$ holds for all indices 
 $i \in \{ 0, 1, \ldots, 4k - 3 \}$ (cf.\ \cite[Lemma 2]{JuanLiu2012}). 
\end{lemma} 

\begin{theorem}\label{maintheorem}
 The relation $\operatorname{an}(C_{4k}) = 2k^2 - 1$ holds for every positive integer $k$. 
\end{theorem}

\begin{proof}
 To prove the reverse inequality shown in \eqref{displayreverse}, we begin with the base case for $ k = 1$. For an arbitrary antipodal 
 labeling of $C_{4}$, since $\operatorname{diam}(C_4) = 2$, we find that any two adjacent vertices $u$ and $v$ necessarily satisfy $| f(u) - 
 f(v) | \geq \operatorname{diam}(C_4) - d(u, v) = 1$, 
 so that $\operatorname{sp}(f) \geq 1$. 
 Similarly, by assigning the label sequence $0$, $1$, $0$, $1$
 in a cyclic order to the vertices of $C_{4}$, we find that adjacent vertices differ by $1$, whereas opposite vertices
 may have equal labels, giving us an antipodal labeling of span $1$. 
 We thus find that $\operatorname{an}(C_{4k}) = 2k^2 - 1$ holds for $k = 1$. 
 We proceed to disregard this base case and to let $k \geq 2$. 

 Our construction, relying on the ordering in \eqref{orderfx}, gives us that 
\begin{equation}\label{evaluatespf} 
 \operatorname{sp}(f) = f(x_{4k-1}) = \sum_{i=0}^{4k-2} f_{i}, 
\end{equation} 
 recalling the definition displayed in \eqref{definefi}. 
 Moreover, since $f$ is antipodal, we have that 
\begin{equation}\label{2kminusdi} 
 f_{i} \geq 2 k - d_i 
\end{equation}
 for all $i \in \{ 0, 1, \ldots, 4k - 2 \}$. 

 We proceed to define the nonnegative integer
\begin{equation}\label{defineepsilon} 
 \varepsilon_{i} = f_{i} + f_{i+1} - k 
\end{equation}
 for $i \in \{ 0, 1, \ldots, 4k-3 \}$, giving the excess over the lower bound given in Lemma \ref{boundfsum}. From the antipodal condition in 
 \eqref{antipodalcondition}, together with 
 the ordering imposed in \eqref{orderfx}, 
 we deduce that 
\begin{equation}\label{withoutabs}
 f(x_{i + 2}) - f(x_{i}) \geq 2 k - d(x_i, x_{i+2}). 
\end{equation}
 From the definition of $\varepsilon_i$ in \eqref{defineepsilon} together
 with \eqref{withoutabs}, we find that 
 $ k + \varepsilon_i \geq 2 k - d(x_i, x_{i+2}) $
 for all indices $i \in \{ 0, 1, \ldots, 4 k - 3 \}$, 
 i.e., so that 
\begin{equation}\label{negepsilon} 
 d(x_{i}, x_{i+2}) \geq k - \varepsilon_{i} 
\end{equation}
 for each $i$ in $ \{ 0, 1, \ldots, 4 k - 3 \}$. 
 In a similar fashion, a combined application of \eqref{2kminusdi}
 and Lemma \ref{lemmadistance} gives us that $ d(x_{i}, x_{i+2}) \leq 4 k - d_i - d_{i+1} \leq 4 k - (2 k - f_i) - 
 (2 k - f_{i+1}) = k +\varepsilon_i$. This, together with \eqref{negepsilon}, gives us that 
\begin{equation}\label{boundabs} 
 \left| k - d(x_i, x_{i+2}) \right| \leq \varepsilon_i. 
\end{equation}

 Now, let $\overline{C}$ denote a copy of $C_{2k}$, 
 with its vertex set written so that 
 $V(\overline{C}) = \{ \overline{v}_{0}, \overline{v}_{1}, \ldots, \overline{v}_{2k-1} \}$. 
 We also write $\overline{d}$ to denote the distance function for $\overline{C}$. 
 We then define a quotient mapping 
\begin{equation}\label{rhocolon} 
 \rho\colon V(C_{4k}) \to V(\overline{C})
\end{equation}
 so that 
\begin{equation}\label{rhovalue}
 \rho(v_{j}) = \overline{v}_{j \bmod {2k}}. 
\end{equation} 
 The function defined via \eqref{rhovalue} thus maps vertices in the same antipodal pair of $C_{4k}$
 to the same vertex in the codomain in \eqref{rhocolon}. 
 Now, define
\begin{equation}\label{taucolon}
 \tau\colon V(\overline{C}) \to V(\overline{C})
\end{equation}
 so that 
 $ 
 \tau(\overline{v}_{j}) = \overline{v}_{(j+k) \bmod {2k}} $ 
 for each $j$ in $ \{ 0, 1, \ldots, 2k-1 \}$. 

 The definition in \eqref{rhovalue} gives us that 
\begin{equation}\label{linedrhorho}
 \overline{d}(\rho(u), \rho(v)) = \min\{ d(u, v), 2 k - d(u, v) \} 
\end{equation}
 for all $u, v \in V(C_{4k})$. 
 Now, since $\tau(\rho(v))$ is the antipode of $\rho(v)$ in $C_{2k}$, we find that 
$ \overline{d}(\rho(u), \tau(\rho(v)))
 = k - \overline{d}(\rho(u), \rho(v)) 
 = k - \min\{ d(u, v), 2 k - d(u, v) \}$, i.e., so that 
\begin{equation}\label{abskduv}
 \overline{d}(\rho(u), \tau(\rho(v))) = \left| k - d(u, v) \right|. 
\end{equation}
 Also, as a consequence of \eqref{linedrhorho}, we have that 
\begin{equation}\label{2kminusduv} 
 \overline{d}(\rho(u), \rho(v)) \leq 2 k - d(u, v) 
\end{equation}
 for arbitrary $u, v \in V(C_{4k})$. 

 Now, define
\begin{equation}\label{defineyi}
 y_{i} = \tau^{\left\lfloor \frac{i}{2} \right\rfloor}\left( \rho(x_{i}) \right) 
\end{equation}
 for $i \in \{ 0, 1, \ldots, 4k-1 \}$. 
 Observe that since $\tau$ is an involution, 
 the power of $\tau$ in the right-hand side of \eqref{defineyi} is determined by the residue class of $i$ modulo $4$.
 Also observe that $\tau$ is an isometry, i.e. 
 so that 
 $ \overline{d}\left( \tau(\overline{u}), \tau(\overline{v}) \right) 
 = \overline{d}\left( \overline{u}, \overline{v} \right) $ 
 for all vertices $\overline{u}$ and $\overline{v}$ in $V(\overline{C})$. 
 
 Using the property such that $\tau$ is an isometry together with the relation in \eqref{abskduv}, we find that $\overline{d}(y_i, y_{i + 
 2}) = \overline{d}(\rho(x_i), \tau(\rho(x_{i+2})))  = |k - d(x_{i}, x_{i+2}) |$, so that an application of \eqref{boundabs} 
 then gives us that 
\begin{equation}\label{metricgaptwo}
 \overline{d}(y_{i}, y_{i+2}) \leq \varepsilon_{i}. 
\end{equation}
 From the definition in \eqref{defineyi}, we see that the vanishing of the superscript on the right of \eqref{defineyi} for $i \in \{ 0, 
 1 \}$ gives us that 
\begin{equation}\label{fromyi} 
 \overline{d}(y_0, y_1) = \overline{d}\big( \rho(x_0), \rho(x_1) \big). 
\end{equation}
 A combined application of \eqref{2kminusdi}, \eqref{2kminusduv}, and 
 \eqref{fromyi} allows us to deduce that 
\begin{equation}\label{firstendpoint}
 \overline{d}(y_0, y_1) \leq 2 k - d_0 \leq f_0. 
\end{equation}
Now, since
 $\left\lfloor \frac{4k-2}{2} \right\rfloor = \left\lfloor \frac{4k-1}{2} \right\rfloor = 2 k - 1$, 
 we find that 
\begin{equation}\label{y4kpair}
 y_{4k-2} = \tau^{2k-1}\left( \rho(x_{4k-2}) \right) \quad \text{and} \quad
 y_{4k-1} = \tau^{2k-1}\left( \rho(x_{4k-1}) \right). 
\end{equation}
 Since $\tau$ is an involution, we obtain from \eqref{y4kpair} that 
\begin{equation}\label{y4kagain}
 y_{4k-2} = \tau\left( \rho(x_{4k-2}) \right) \quad \text{and} \quad
 y_{4k-1} = \tau\left( \rho(x_{4k-1}) \right). 
\end{equation}
 Since $\tau$ is an isometry, this and \eqref{y4kagain} together give us that 
\begin{align}
\begin{split}
 \overline{d}(y_{4k-2}, y_{4k-1})
 & = \overline{d}(\tau(\rho(x_{4k-2})), \tau(\rho(x_{4k-1}))) \\
 & = \overline{d}(\rho(x_{4k-2}), \rho(x_{4k-1})). 
\end{split}\label{isometryy4k}
\end{align}
 A combined application of \eqref{2kminusduv} and \eqref{isometryy4k} then gives us that 
\begin{equation}\label{aftertauy4k}
 \overline{d}(\rho(x_{4k-2}), \rho(x_{4k-1})) \leq 2 k - d_{4k-2}. 
\end{equation}
 Now, a combined application of \eqref{2kminusdi}, \eqref{isometryy4k}, and \eqref{aftertauy4k} gives us that 
\begin{equation}\label{secondendpoint} 
 \overline{d}(y_{4k-2}, y_{4k-1}) \leq f_{4k-2}. 
\end{equation}

 Consider the tuple 
\begin{equation}\label{displayW}
 W = \big( y_0, y_2, \ldots, y_{4k-2}, y_{4k-1}, y_{4k-3}, \ldots, y_{1}, y_{0} \big) 
\end{equation}
 consisting of vertices in the underlying set of the metric space $(V(\overline{C}), \overline{d})$, noting the first and final entries of the 
 tuple in \eqref{displayW} are the same. 
 We define the \emph{metric length} $\operatorname{met}(W)$ of $W$ (as opposed to the length of $W$ as a tuple) 
 as the sum of the distances between the consecutive entries of $W$. 
 From \eqref{metricgaptwo}, we find that 
\begin{equation}\label{firstepsilonsum} 
 \sum_{\substack{ 0\leq i \leq 4k-4 \\ \text{$i$ even} }} \overline{d}(y_i, y_{i+2}) 
 \leq \sum_{\substack{0\leq i \leq 4k-4 \\ \text{$i$ even}}} \varepsilon_i 
\end{equation}
 and that 
\begin{equation}\label{secondepsilonsum} 
 \sum_{\substack{ 1 \leq i \leq 4 k - 3 \\ \text{$i$ odd} }} \overline{d}(y_{i+2}, y_{i}) 
 \leq \sum_{\substack{ 1 \leq i \leq 4 k - 3 \\ \text{$i$ odd}}} \varepsilon_{i}. 
\end{equation}
 The definition in \eqref{displayW} and the definition of $\operatorname{met}(W)$ give us that 
\begin{equation}\label{firstmetbound} 
 \operatorname{met}(W) = 
 \overline{d}(y_{4k-2}, y_{4k-1}) 
 + \overline{d}(y_{1}, y_{0}) 
 + \sum_{\substack{ 0\leq i \leq 4k-4 \\ \text{$i$ even} }} \overline{d}(y_i, y_{i+2}) + \sum_{\substack{ 1 \leq i \leq 4 k - 3 \\ \text{$i$ odd} }} \overline{d}(y_{i+2}, y_{i}). 
\end{equation}
 A combined application of \eqref{firstendpoint}, 
 \eqref{secondendpoint}, \eqref{firstepsilonsum}, 
 \eqref{secondepsilonsum}, and \eqref{firstmetbound} gives us that 
\begin{equation}\label{metWabove}
 \operatorname{met}(W) \leq f_0 + f_{4k-2} + \sum_{i=0}^{4k-3} \varepsilon_i.
\end{equation}
 We proceed to determine a lower bound for $\operatorname{met}(W)$. 

 Recall that $\overline{C}$ is an isomorphic copy of $C_{2k}$. Also, recall that $y_{i}$, for a given index $i \in \{ 0, 1, \ldots, 4 k - 1 \}$ and 
 according to the definition in \eqref{defineyi}, is obtained by taking a vertex $x_{i}$ in $V(C_{4k})$ (recalling \eqref{VC4k}), mapping 
 $x_i$ to an element in the codomain $V(\overline{C})$ in \eqref{rhocolon}, and then mapping the resultant vertex to an element in 
 the codomain $V(\overline{C})$ in \eqref{taucolon}, under possibly repeated applications of the involution $\tau$. Also observe that 
\begin{equation}\label{ztauz}
 \tau(z) \neq z
\end{equation}
 for each element $z$	 in the domain $V(\overline{C})$ of $\tau$, i.e., since $\tau$ maps $z$ to the antipodal vertex of $z$. We claim 
 that the sequence $y_{0}$, $y_{1}$, $\ldots$, $y_{4k-1}$ contains at least $k$ distinct vertices of $\overline{C}$, and this is 
 demonstrated below. 

 Fix a vertex $z$ in $V(\overline{C})$. If $y_{i} = z$ for some index $i \in \{ 0, 1, \ldots, 4 k - 1 \}$, then, according to the definition 
 in \eqref{defineyi} (and recalling \eqref{ztauz} and that $\tau$ is an involution), we have that $\rho(x_{i}) \in \{ z, \tau(z) \}$. Each vertex 
 of $\overline{C}$ has precisely two preimages with respect to $\rho$. So, the set $\rho^{-1}(\{ z, \tau(z) \})$ contains precisely four 
 vertices of $C_{4k}$ (recalling the domain in \eqref{rhocolon}). Since the expressions among $x_0$, $x_1$, $\ldots$, and $ x_{4k-1}$ 
 are pairwise unequal (recalling \eqref{VC4k}), any choice of $z \in V(\overline{C})$ can occur at most four times among expressions of 
 the form $y_{i}$ for $i \in \{ 0, 1, \ldots, 4 k - 1 \}$. This shows that 
\begin{equation}\label{atleastky}
 | \{ y_{i} : i \in \{ 0, 1, \ldots, 4 k - 1 \} \} |
 \geq \frac{4k}{4} = k. 
\end{equation}
 From \eqref{displayW} and \eqref{atleastky} together, we find that $W$ contains at 
 least $k$ distinct vertices. 

 We claim that $W$ is of metric length at least $2k-2$. To show this, we begin by forming a closed walk along the edges of $C_{2k}$, by 
 replacing each pair of consecutive entries in $W$ with a shortest path between these vertices, and by forming a walk from the 
 consecutive paths formed. Observe that the length of this resulting walk is equal to $\operatorname{met}(W)$. Now, let $H$ denote 
 the connected subgraph consisting of all vertices and edges traversed by the walk. If $H = C_{2k}$ (letting $\overline{C}$ be identified 
 with $C_{2k}$ for convenience), then each edge of $C_{2k}$ is used at least once, so that the length of the walk is at least $2k$, and 
 hence at least $2k-2$. If $H \neq C_{2k}$, then, since $H$ is (by construction) a connected subgraph of $C_{2k}$, we find that $H$ is a 
 path. Since $W$ contains at least $k$ distinct vertices (recalling \eqref{atleastky}), we then find that $H$ has at least $k - 1$ edges. 
 Each edge of $H$ is a bridge. Since the walk we have constructed is closed and uses every edge of $H$, each edge of $H$ is necessarily
 traversed at least twice. Moreover, since $H$ contains at least $k$ vertices, it has at least $k-1$ edges. Consequently, the length of the 
 walk is at least $2 | E(H) | \geq 2(k-1)$, and thus we thus have the desired relation such that
\begin{equation}\label{metWlower}
 \operatorname{met}(W) \geq 2 k - 2. 
\end{equation} 
 A combined application of \eqref{metWabove} and \eqref{metWlower} then gives us that 
\begin{equation}\label{combinemetW} 
 f_{0} + f_{4k-2} + \sum_{i=0}^{4k-3} \varepsilon_{i} \geq 2 k - 2. 
\end{equation}
 From \eqref{combinemetW} and the definition of $\varepsilon_{i}$ in \eqref{defineepsilon}, we see, 
 with the use of a reindexing argument, that 
\begin{align}
\begin{split}
 \sum_{i=0}^{4k-3} \varepsilon_{i} 
 & = \sum_{i=0}^{4k-3} (f_{i} + f_{i+1} - k) \\ 
 & = 2 \sum_{i=0}^{4k-2} f_{i} - f_{0} - f_{4k-2} - (4k-2) k.
\end{split}\label{reindexing}
\end{align}
 Using \eqref{evaluatespf} and \eqref{reindexing} together, we then obtain that 
\begin{equation}\label{reintroducesp} 
 f_{0} + f_{4k-2} + \sum_{i=0}^{4k-3} \varepsilon_{i} = 2 \operatorname{sp}(f) - (4k-2) k. 
\end{equation}
 The relations in \eqref{combinemetW} and \eqref{reintroducesp} together give us that $ 2 \operatorname{sp}(f) - (4 k - 2) k \geq 2 k - 
 2$, which, in turn, implies that 
\begin{equation}\label{spflower} 
 \operatorname{sp}(f) \geq 2 k^2 - 1. 
\end{equation}
 Since $f$ was chosen as an arbitrary antipodal labeling of $C_{4k}$, we can conclude from \eqref{spflower} that 
 $\operatorname{an}(C_{4k}) \geq 2k^2 - 1$. This and the Juan--Liu upper bound in \eqref{C4kbounds} 
 \cite[Theorem 9]{JuanLiu2012}
 then give us the desired result. 
\end{proof}

\section{Conclusion}
 As a natural way of extending the techniques and main result in this paper, one might consider the problem of evaluating 
 $\operatorname{an}(G)$ for the case whereby $G$ is in a natural family of graphs extending cycle graphs.  We encourage the pursuit of 
 research endeavors based on this. 

\subsection*{Acknowledgements}
 The author acknowledges extensive interactions with GPT-5.6 Pro during the exploratory and proof-development stages of this work. All 
 AI-generated suggestions were substantially revised, corrected, and independently verified by the author, who assumes full responsibility 
 for the mathematical content.

\bibliographystyle{plain}
\bibliography{seprefe}

\end{document}